\documentclass[12pt]{iopart}
\expandafter\let\csname equation*\endcsname\relax
\expandafter\let\csname endequation*\endcsname\relax

\usepackage{amsmath,amsfonts,amssymb,amsthm}
\usepackage{graphicx}
\usepackage{color,colordvi}
\usepackage{hyperref}

\def\softd{{\leavevmode\setbox1=\hbox{d}%
\hbox to 1.05\wd1{d\kern-0.4ex{\char039}\hss}}}
\def\softt{{\leavevmode\setbox1=\hbox{t}%
\hbox to \wd1{t\kern-0.6ex{\char039}\hss}}}

\newcommand{\R}{\mathbb{R}}
\newcommand{\Z}{\mathbb{Z}}
\newcommand{\CC}{\mathcal{C}}
\newcommand{\HH}{\mathcal{H}}
\newcommand{\OO}{\mathcal{O}}
\newcommand{\cS}{\mathcal{S}}
\newcommand{\ee}{\mathrm{e}}
\newcommand{\D}{\mathrm{d}}

\newtheorem{theorem}{Theorem}[section]

\newtheorem{conjecture}{Conjecture}[section]
\newtheorem{remark}{Remark}[section]
\newtheorem{remarks}{Remarks}[section]

\begin{document}

\title[Point interaction optimization]
{An optimization problem for finite point interaction families}

\author{Pavel Exner}
\address{Nuclear Physics Institute, Academy of Sciences of the Czech Republic,
Hlavn\'{i} 130, 25068 \v{R}e\v{z} near Prague, Czech Republic}
\address{Doppler Institute, Czech Technical University, B\v{r}ehov\'{a} 7, 11519 Prague, Czech Republic}
\ead{exner@ujf.cas.cz}

\begin{abstract}
We consider the spectral problem for a family of $N$ point interactions of the same strength confined to a manifold with a rotational symmetry, a circle or a sphere, and ask for configurations that optimize the ground state energy of the corresponding singular Schr\"odinger operator. In case of the circle the principal eigenvalue is sharply maximized if the point interactions are distributed at equal distances. The analogous question for the sphere is much harder and reduces to a modification of Thomson problem; we have been able to indicate the unique maximizer configurations for $N=2,\,3,\,4,\,6,\,12$. We also discuss the optimization for one-dimensional point interactions on an interval with periodic boundary conditions. We show that the equidistant distributions give rise to maximum ground state eigenvalue if the interactions are attractive, in the repulsive case we get the same result for weak and strong coupling and we conjecture that it is valid generally.
\end{abstract}

\pacs{02.30.Tb, 03.65.Ge, 03.65.Db}

%
\vspace{2pc} \noindent{\it Keywords}: point interactions, ground state, optimization, rotational symmetry

%
\submitto{\JPA}

%
%
%

\section{Introduction} \label{s: intro}

The search for geometric configurations that make a given spectral quantity optimal is a trademark topic in mathematical physics with a century long tradition reaching back to the Faber's and Krahn's celebrated proof of Lord Rayleigh's conjecture about the lowest drum tone \cite{Fa23,Kr25}. Without recalling the long and rich history we note a recent wave of interest to the topic concerning the ground state of Robin billiards \cite{FK15,AFK17,KL18,KL19} and of leaky quantum structures \cite{EL17,EL18,EK19}, in a sense the one- and two-sided aspects of the same problem.

This brings to mind similar optimization problems for families of point interactions somehow neglected in those recent efforts. More then a decade ago the ground state optimization problem was solved for point interactions placed at the vertices of an equilateral polygon \cite{Ex06}, but the work mentioned above inspires another question, namely to consider point interactions placed at a manifold of maximum symmetry, a circle or a sphere, and to ask about optimal configurations. We note that dealing with a finite family of point interactions is easier than the optimization for billiards and leaky structures since one has to deal with matrices instead of integral operators, however, some aspects of discrete optimization may be on the contrary more involved; to illustrate this claim one can compare the proof in \cite{Ex06} with that used in the continuous analogue of that particular problem \cite{EHL06}.

The two indicated problems differ substantially in the degree of difficulty. In case of point interactions placed at a circle, both in two and three dimensions, it is easy to guess the optimal configuration and, maybe with some effort, to confirm the guess. The optimization on a sphere, on the other hand, is a rather hard task which reminds one about another problem more than a century old, a search for configurations of point charges on a sphere that minimize the electrostatic energy \cite{Th04}. Our question will be reduced to a similar minimization problem with a different `potential' and using results from algebraic combinatorics we will be able to find an answer in five cases with a low number of point interactions leaving the other situations open.

We also include into the discussion the case of one-dimensional point interactions which are less singular than their two- and three-dimensional counterparts. There is no zero measure subset analogous to a circle or a sphere here, but we can consider such interactions on a loop, in other words, on an interval with periodic boundary conditions. As long as the interactions are attractive the ground state energy is again maximized by configurations with a maximum symmetry as can be shown in the same, and even a bit simpler way. The case of repulsive $\delta$ potentials is more complicated and in this paper we provide a partial result only.

\section{A warm up: attractive point interactions on a loop}
\label{s:1Dloop}
\setcounter{equation}{0}

To begin with, let us investigate the Hamiltonian describing $N$ attractive point interactions, $N\ge 2$, of identical strength on a loop; in view of the scaling properties of the system we may assume without loss of generality that the loop length is $2\pi$. In other words, we consider $\HH = L^2(I)$ with $I=(0,2\pi)$ and the operator written formally as
\begin{equation} \label{1Dham}
     H_{\alpha,Y} = -\frac{\D^2}{\D x^2} + \alpha \sum_{n=1}^N \delta(x-y_n)
\end{equation}
with $\alpha<0$ and $Y:=\{y_j\}_{j=1}^N$ such that $0<y_1<\cdots<y_N<2\pi$, which means that $H_{\alpha,Y}$ acts as a negative Laplacian on functions from $H^2(I)$ satisfying the conditions
\begin{subequations}
   \begin{align}
     & \psi(y_j+) = \psi(y_j-) =:\psi(y_j)\,, \quad \psi'(y_j+)-\psi'(y_j-)= \alpha \psi(y_j)\,, \quad j=1,\dots,N\,, \label{1Ddelta} \\
     & \psi(2\pi-) = \psi(0+)\,, \quad \psi'(2\pi-) = \psi'(0+)\,. \label{periodic}
   \end{align}
\end{subequations}
Alternatively, $H_{\alpha,Y}$ is the unique self-adjoint operator associated with the quadratic form
\begin{equation} \label{1Dhamform}
     q_{\alpha,Y}:\; q_{\alpha,Y}(\psi) = \|\psi'\|^2 + \alpha \sum_{j=1}^N |\psi(y_j)|^2
\end{equation}
defined on periodic functions from $H^1(I)$. It is easy to see that the spectrum of the operator is purely discrete, $\sigma(H_{\alpha,Y}) = \{\lambda_n(\alpha,Y)\}_{n=1}^\infty$, arranged conventionally in the ascending order, and the ground state eigenvalue $\lambda_1(\alpha,Y)$ is simple and negative. Moreover, it follows easily from \eqref{1Dhamform} and the minimax principle that $\lambda_1(\alpha,Y)<0$ for $\alpha<0$; we are interested in the sets $Y$ that optimize the ground state.

To this aim we need the resolvent of the free Hamiltonian $H_0\equiv H_{0,Y}$ which is at the energy $-\kappa^2$ a convolution-type integral operator with the kernel
\begin{equation} \label{1Dgreen}
     G_{i\kappa}(x-y) = \frac{\cosh(\kappa(\pi-|x-y|)}{2\kappa \sinh \pi\kappa}\,.
\end{equation}
To derive this formula one can use the explicit knowledge of the spectrum: the eigenvalues $m^2$ correspond to the eigenfunctions $\eta_m(x) = \frac{1}{\sqrt{2\pi}} \ee^{imx},\: m\in\Z$. Consequently, we have
$$ 
     G_{i\kappa}(x-y) = \sum_{m\in\Z} \frac{\eta(x)\overline\eta(y)}{m^2+\kappa^2}
$$ 
and the sum is easily evaluated \cite[5.4.3.4]{PBM}. It is important to note that the function $G_{i\kappa}$ is strictly convex in the interval $(0,\pi)$. In straightforward analogy with \cite[Thm.~II.2.1.1]{AGHH} one can the express the sought resolvent of $H_{\alpha,Y}$ by means of Krein's formula
\begin{subequations}
   \begin{align}
     & (H_{\alpha,Y}+\kappa^2)^{-1} = G_{i\kappa} + \sum_{j,j'=1}^N [\Gamma_{\alpha,Y}(i\kappa)]_{jj'}^{-1} \big( \overline{G_{i\kappa}(\cdot - y_{j'})}, \cdot \big) G_{i\kappa}(\cdot - y_j)\,, \label{1Dkrein1} \\[.3em]
     & \Gamma_{\alpha,Y}(i\kappa) = - [-\alpha^{-1}\delta_{jj'} + G_{i\kappa}(y_j-y_{j'})]_{j,j'=1}^N\,. \label{1Dkrein2}
   \end{align}
\end{subequations}
This implies that negative eigenvalues of $\sigma(H_{\alpha,Y})$ are obtained as roots of the equation $\det \Gamma_{\alpha,Y}(i\kappa)=0$, in particular, $\lambda_1(\alpha,Y)$ correspond corresponds to the value of $\kappa$ at which the smallest eigenvalue of $\Gamma_{\alpha,Y}(i\kappa)$ vanishes. Indeed, according to \cite{Kr53} the eigenvalues are for any fixed $\alpha<0$ continuously decreasing functions of the energy, and thus also of $-\kappa$ in the interval $(-\infty,0)$. The optimization, or more specifically \emph{maximization} of the ground state is then equivalent to identifying sets $Y$ for which the said smallest eigenvalue of $-\Gamma_{\alpha,Y}(i\kappa)$ is \emph{maximal}.

\begin{theorem}\label{th:1Dmaxim}
Put $\tilde{Y}:= \{ \frac{\pi}{N}(2j-1)\}_{j=1}^N$, then for any $N$-point set $Y$ and any $\alpha<0$ we have
\begin{equation} \label{1Dmaxim}
     \lambda_1(\alpha,Y) \le \lambda_1(\alpha,\tilde{Y})
\end{equation}
and the inequality is sharp if and only if $Y$ is not congruent with $\tilde{Y}$.
\end{theorem}
\begin{proof}
As mentioned above, the eigenvalue $\lambda_1(\alpha,Y)$ is simple and same is true for the smallest eigenvalue $\mu_1(\alpha,Y)$ of $\Gamma_{\alpha,Y}(i\kappa)$. The corresponding eigenfunction can be chosen positive. If $Y=\tilde{Y}$ the system is invariant with respect cyclic shifts by $\frac{2\pi}{N}$ which means that the respective eigenfunction of $\Gamma_{\alpha,\tilde{Y}}(i\kappa)$ is $\varphi_1= \frac{1}{\sqrt{N}}(1,\dots,1)$.

It is useful to extend $Y$ periodically to the real line identifying the points modulo $2\pi$ and note that then all the distances $d_{jj'}= d(y_j, y_{j'}) \in (0,\pi]$. By minimax principle and \eqref{1Dkrein2} we have
\begin{equation} \label{1Dineq}
     \mu_1(\alpha,Y) \le (\varphi_1,\Gamma_{\alpha,Y} \varphi_1) = \frac{1}{\alpha} - \frac{1}{N} \sum_{j,j'=1}^N g_{i\kappa}(d_{jj'})\,,
\end{equation}
where $g_{i\kappa}(d_{jj'}) := G_{i\kappa}(y_j-y_{j'})$. Our aim is to show that the right-hand side does not exceed $\mu_1(\alpha,\tilde{Y}) = (\varphi_1,\Gamma_{\alpha,\tilde{Y}} \varphi_1)$ being sharply smaller if $Y$ is not congruent with $\tilde{Y}$. This will be for sure true if
$$ 
     \sum_{j,j'=1}^N g_{i\kappa}(d_{jj'}) > \sum_{j,j'=1}^N g_{i\kappa}(\tilde{d}_{jj'})
$$ 
would hold for all $\kappa>0$ unless $\{d_{jj'}\} = \{\tilde{d}_{jj'}\}$ holds between the sets of distances determining $Y$ and $\tilde{Y}$, respectively. In fact, we can avoid double counting and consider only summation over $j<j'$. Rearranging the summation, we have to check that
$$ 
     F(\{d_{jj'}\}) := \sum_{m=1}^{[N/2]} \sum_{|j-j'|=m} \big[ g_{i\kappa}(d_{jj'}) - g_{i\kappa}(\tilde{d}_{jj'}) \big] > 0
$$ 
unless $\{d_{jj'}\} = \{\tilde{d}_{jj'}\}$. We recall that $g_{i\kappa}$ is strictly convex in $(0,\pi)$; this allows us to apply Jensen's inequalityto the inner sums which yields
\begin{equation} \label{1Djensen}
     F(\{d_{jj'}\}) \ge \sum_{m=1}^{[N/2]} \nu_m \bigg[ g_{i\kappa}\Big( \frac{1}{\nu_m} \sum_{|j-j'|=m} d_{jj'} \Big) - g_{i\kappa}(\tilde{d}_{1,1+m}) \bigg]\,,
\end{equation}
where $\nu_m$ is the number of the distinct loop arcs between the points $y_j$ and $y_{j+m}$ for $j=1,\dots,n$, explicitly
$$ 
     \nu_m := \left\{ \begin{array}{ccl} N \quad & \dots & \quad m=1,\dots,\big[ \frac12(N-1)\big] \\[.3em]
     \frac12\,N \quad & \dots & \quad m= \frac12\,N \;\; \text{for}\;\; N\;\text{even} \end{array} \right.
$$ 
and in addition, the inequality in \eqref{1Djensen} is sharp unless all the summands in the argument are mutually equal. Finally, the said sum is easily computed,
$$ 
     \frac{1}{\nu_m} \sum_{|j-j'|=m} d_{jj'} = \frac{2\pi m}{N} = \tilde{d}_{1,1+m}\,,
$$ 
hence the right-hand side of \eqref{1Djensen} is zero which concludes the proof.
\end{proof}

\section{Point interactions on a circle} \label{s:circle}
\setcounter{equation}{0}

Let now $\CC$ be a circle of unit radius in $\R^\nu,\: \nu=2,3$, and consider a system of $N$ point interactions of equal strength placed at $\CC$. The assumed size of $\CC$ does not restrict the generality of our conclusions due to the scaling properties of the corresponding Hamiltonians. We denote again by $Y:=\{y_j\}_{j=1}^N \subset \CC$ the interaction support and use the symbol $H_{\alpha,Y}$ for the singular Schr\"odinger operators with point interactions in $L^2(\R^\nu)$. The way to construct them is well known \cite[Secs.~II.1 and II.4]{AGHH}: they are defined on function from $H^2(\R^\nu\setminus Y)$ acting there as the negative Laplacian, while at the points of $Y$ boundary conditions are imposed relating the generalized boundary values at each $y_j$, the coefficient at the singularity (logarithmic for $\nu=2$, first-order pole for $\nu=3$) and the next term in the
expansion. The relations are linear and the coefficient $\alpha\in\R$ appearing in them characterizes the coupling strength, however, it is not the usual coupling constant as seen from the fact that the point interaction absence corresponds formally to $\alpha=\infty$.

It is also well known that $\sigma_\mathrm{ess}(H_{\alpha,Y})= [0,\infty)$ and the existence of discrete spectrum depends on the dimension: $\sigma_\mathrm{disc}(H_{\alpha,Y})\ne\emptyset$ holds always if $\nu=2$ while for $\nu=3$ this requires $\alpha<\alpha_\mathrm{crit}^Y$ where superscript indicates that the value of the critical coupling depends on $Y$; we have $\alpha_\mathrm{crit}=0$ for a single point interaction, for $N=2$ already we have $\alpha_\mathrm{crit}>0$. Since our aim is again to find a configuration $Y$ that maximizes the lowest eigenvalue, in the three-dimensional situation we assume $\alpha<\alpha_\mathrm{crit}:= \inf_Y \alpha_\mathrm{crit}^Y$, and for brevity we also set $\alpha_\mathrm{crit}=\infty$ if $\nu=2$. The spectral problem can be again conveniently solved using Krein's formula which looks like \eqref{1Dkrein1} with the matrix \eqref{1Dkrein2} replaced by
\begin{equation} \label{nuDkrein}
     \Gamma_{\alpha,Y}(i\kappa) = \big[(\alpha-\xi_{i\kappa})\delta_{jj'} - (1-\delta_{jj'}) G_{i\kappa}(y_j-y_{j'})\big]_{j,j'=1}^N\,,
\end{equation}
where we have
 \begin{equation} \label{nuDgreen}
 G_{i\kappa}(y_j-y_{j'}) = \left\{ \begin{array}{ccc} \frac{1}{2\pi}
 K_0(\kappa|y_i-y_j|) &\quad\dots\quad& \nu=2 \\ [.3em]
 \frac{\ee^{-\kappa|y_i-y_j|}}{4\pi|y_i-y_j|}
 &\quad\dots\quad& \nu=3 \end{array} \right.
 \end{equation}
and the regularized value of Green's function at the interaction site is
 $$
 \xi_{i\kappa} = \left\{ \begin{array}{ccc} -\frac{1}{2\pi}
 \left(\ln\frac{\kappa}{2} +\gamma_\mathrm{E} \right)
 &\quad\dots\quad& \nu=2 \\ [.3em]
 -\frac{\kappa}{4\pi} &\quad\dots\quad& \nu=3 \end{array} \right.
 $$
where $\gamma_\mathrm{E} = -\psi(1) \approx 0.57721$ is the Euler-Mascheroni constant. The formula \eqref{nuDgreen} allows us to find the discrete spectrum of $H_{\alpha,Y}$ through solutions of the equation $\det \Gamma_{\alpha,Y}(i\kappa)=0$ and for the ground state eigenvalue we have the following result:
\begin{theorem}\label{th:nuDmaxim}
For any $N$-point set $Y\subset\CC$ and any $\alpha<\alpha_\mathrm{crit}$ we have
\begin{equation} \label{nuDmaxim}
     \lambda_1(\alpha,Y) \le \lambda_1(\alpha,\tilde{Y})\,,
\end{equation}
where $\tilde{Y}$ is the family of vertices of a regular $N$-polygon, and the inequality is sharp if and only if $Y$ and $\tilde{Y}$ are not congruent.
\end{theorem}
\begin{proof}
The first part of the argument is analogous to the optimization of point interaction `necklaces' performed in \cite{Ex06}, hence we describe it only briefly. As in the previous section we seek the configuration maximizing the lowest eigenvalue of $\Gamma_{\alpha,Y}(i\kappa)$, in other words, we want to prove that
 \begin{equation} \label{Gammaineq}
 \mu_1(\alpha,Y) = \min \sigma(\Gamma_{\alpha,Y}(i\kappa)) < \min \sigma(\Gamma_{\alpha,\tilde{Y}}(i\kappa)) = \mu_1(\alpha,\tilde{Y})(i\kappa))
 \end{equation}
holds for all $\kappa>0$. We denote $g_{i\kappa}(\ell_{jj'}) := G_{i\kappa}(y_j-y_{j'})$ where the different symbol indicates that in contrast to \eqref{1Dineq} the distance is not measured over the perimeter of the circle but over the respective chord, i.e. as the Euclidean distance of the points $y_j$ and $y_{j'}$ in $\R^\nu$. The lowest eigenvalue of $H_{\alpha,Y}$ is simple and the corresponding eigenfunction can be chosen positive. Furthermore, for $Y=\tilde{Y}$ it has the symmetry with respect to the discrete group of rotations by multiples of the angle $\frac{2\pi}{N}$ which means that the eigenvector corresponding to $\mu_1(\alpha,\tilde{Y})(i\kappa)$ is again $\varphi_1= \frac{1}{\sqrt{N}}(1,\dots,1)$. In combination with minimax principle this allows us to write
\begin{equation} \label{nuDineq}
     \mu_1(\alpha,Y) \le (\varphi_1,\Gamma_{\alpha,Y} \varphi_1) = \alpha -\xi_{i\kappa} - \frac{1}{N} \sum_{j\ne j'} g_{i\kappa}(\ell_{jj'})\,.
\end{equation}
According to \eqref{nuDgreen} the function $g_{i\kappa}(\cdot)$ is strictly convex which implies an inequality analogous to \eqref{1Djensen}, namely
\begin{equation} \label{nuDjensen}
     F(\{\ell_{jj'}\}) \ge \sum_{m=1}^{[N/2]} \nu_m \bigg[ g_{i\kappa}\Big( \frac{1}{\nu_m} \sum_{|j-j'|=m} \ell_{jj'} \Big) - g_{i\kappa}(\tilde{\ell}_{1,1+m}) \bigg]\,,
\end{equation}
sharp unless $Y$ and $\tilde{Y}$ are congruent. In contrast to the previous section, however, we cannot hope that the right-hand side would vanish. Instead we want to prove that it is positive, or non-negative at worst. Since $g_{i\kappa}(\cdot)$ is monotonously decreasing in $(0,\infty)$, it is sufficient to check the inequalities
 \begin{equation} \label{nuDchords}
 \tilde\ell_{1,m+1} \ge \frac{1}{\nu_n} \sum_{|j-j'|=m} \ell_{jj'}
 \end{equation}
for $m=1,\dots,\big[\frac{N}{2}\big]$. Until now the argument followed closely the reasoning of \cite[Sec.~2]{Ex06}. The geometry of the problem there, however, is different from the present one. In order to prove inequality \eqref{nuDchords} for points on the circle we express the corresponding chord lengths in terms of their angular distances, $\ell_{jj'} = 2\sin \frac12\phi_{jj'}$ and $\tilde\ell_{1,m+1} =2\sin\frac{\pi m}{N}$. We note that $\frac12\phi_{jj'} \in \big(0,\frac{\pi}{2}\big]$ which allows to employ Jensen's inequality again, now for \emph{concave} functions, which yields
 $$
\sum_{|j-j'|=m} \frac{1}{\nu_m}\, 2\sin \frac12\phi_{jj'} \le 2\sin\Big(\sum_{|j-j'|=m} \frac{1}{\nu_m}\, \frac12\phi_{jj'} \Big) = 2\sin\frac{\pi m}{N} = \tilde\ell_{1,m+1}\,,
 $$
where the inequality is sharp unless all the angles $\phi_{j,j+m}$ are the same. This proves \eqref{nuDchords}, and \emph{eo ipso} the claim of the theorem.
\end{proof}

\begin{remark} \label{critcircle}
{\rm The secular equation $\det \Gamma_{\alpha,Y}(i\kappa)=0$ in combination with \eqref{nuDkrein} shows that the spectrum is monotonic with respect to $\alpha$, i.e. we have $H_{\alpha,Y}\le H_{\alpha',Y}$ for $\alpha<\alpha'$. This shows \emph{a posteriori} that $\alpha_\mathrm{crit} = \tilde\alpha_\mathrm{crit}$, the critical coupling for the fully symmetric set $\tilde{Y}$. Note that one can extend slightly the above reasoning to conclude that $\sigma_\mathrm{disc}(H_{\alpha_\mathrm{crit},Y}) \ne \emptyset$ holds for any $Y$ which is not congruent with $\tilde{Y}$.}
\end{remark}

\section{Point interactions on a sphere} \label{s:sphere}
\setcounter{equation}{0}

Let us now pass to a much harder problem and consider $N$ point interactions in $\R^3$ confined to the surface of a unit sphere $\cS$. We again denote the interaction support by $Y$ and ask for a configuration $\tilde{Y}$ which maximizes the ground state eigenvalue $\lambda_1(\alpha,Y)$ of $H_{\alpha,Y}$; the question makes sense if $\alpha<\alpha_\mathrm{crit}$. It is natural to expect that $\tilde{Y}$ will exhibit the maximum possible symmetry. The relation between $\lambda_1(\alpha,Y)$ and the lowest eigenvalue of $\Gamma_{\alpha,Y}(i\kappa)$ is the same as in the previous section, however, the other ingredient of the argument that allowed us to pass to \eqref{nuDineq}, namely the possibility to choose $\varphi_1= \frac{1}{\sqrt{N}}(1,\dots,1)$ as the eigenfunction corresponding to the said eigenvalue, can be justified in particular cases only.

Recall that in the circle case it was related to a particular symmetry, the invariance of $\tilde{Y}$ with respect to the group of discrete rotations by multiples of $\frac{2\pi}{N}$. For point on a sphere, such a symmetry is obvious for the smallest values, $N=2,3$, where the $\tilde{Y}$ consists of antipodal points and of three points equally spaced at the equator, respectively. For larger values the needed symmetry exists for $N=4,6,8,12,20$ if the points of $\tilde{Y}$ are identified with the vertices of one the five Platonic solids, tetrahedron, cube, octahedron, dodecahedron an icosahedron, respectively. In those cases we use \eqref{nuDineq} and ask about the maximizer of it right-hand side, in other words, to prove that
 \begin{equation} \label{3Dgreen_comp}
\sum_{j\ne j'} g_{i\kappa}(\tilde\ell_{jj'}) \le \sum_{j\ne j'} g_{i\kappa}(\ell_{jj'})
 \end{equation}
holds, sharply so if the two sets are not congruent, where $\tilde\ell_{jj'}$ and $\ell_{jj'}$ are again the Euclidean distances between the respective points of $\tilde{Y}$ and $Y$ and $g_{i\kappa}$ refers to the three-dimensional resolvent kernel \eqref{nuDgreen}.

\begin{theorem}\label{th:3Dmaxim}
Let the number of points of $Y$ be $N=2,3,4,6,12$ and $\alpha<\alpha_\mathrm{crit}$, then we have
\begin{equation} \label{3Dmaxim}
     \lambda_1(\alpha,Y) \le \lambda_1(\alpha,\tilde{Y})\,,
\end{equation}
where $\tilde{Y}$ is the corresponding fully symmetric set described above; the inequality is sharp if and only if $Y$ and $\tilde{Y}$ are not congruent.
\end{theorem}
\begin{proof}
To begin with, we observe that the task is of the type of \emph{Thomson problem} \cite{Th04} about an optimal distribution of $N$ charges on the surface of a sphere, instead of the Coulomb potential $(y_j,y_{j'}) \mapsto e|y_j-y_{j'}|^{-1}$ we have here the function $g_{i\kappa}$. Despite the fact that the plum-pudding model from which it came proved soon to be physically untenable, the question posed to a hard mathematical challenge, still far from being fully solved more than a century after it was formulated \cite{Twiki}. To illustrate this claim, recall that the problem was included in the `new Hilbert problems' for the present century \cite{Sm98}, and that a (computer assisted) proof for $N=5$ was found only recently \cite{Sch13}.
Furthermore, the original problem inspired a vast activity in algebraic combinatorics, see \cite{CK07}, \cite{BB09} and references therein, where the analogous question is asked in other dimensions, for other manifolds, and for other `potentials'.

We employ one of those generalizations. Given an $N$-point set $Y\subset\cS$ we call it an \emph{$M$-spherical design} if for any polynomial $x\mapsto p(x)$ on $\R^3$ of total degree at most $M$ one can replace its average over the sphere by the average over $Y$, in other words, $\int_{\cS} p(x)\mathrm{d}x = \frac{1}{N} \sum_{j=1}^N p(y_j)$ holds. Let further $m$ be the number of \emph{different} inner products between distinct points of $Y$, then the set is a \emph{sharp configuration} if it is a $2m\!-\!1$ spherical design. The deep result of Cohen and Kumar \cite[Thm.~1.2]{CK07} says  that sharp configurations are \emph{universally optimal} meaning that they minimize \emph{any} potential energy $f:[0,4]\to\R$ which is \emph{completely monotonic}, $(-1)^k  f^{(k)} \geq 0$ for all $k\ge 0$, i.e.
 $$
\sum_{j\ne j'} f(\tilde\ell_{jj'}^{\,2}) \le \sum_{j\ne j'} f(\ell_{jj'}^2)\,,
 $$
sharply so if the complete monotonicity is strict.

The function $g_{i\kappa}(\cdot)$ factorizes into a product of two strictly completely monotonic functions, hence it is also strictly completely monotonic, and it is easy to check that the same is true for $g_{i\kappa} \big(\sqrt{\cdot})\big)$. This means that \eqref{3Dgreen_comp} is valid, sharply unless $Y$ and $\tilde{Y}$ are congruent, for any sharp configuration in $\R^3$. According to \cite[Table~1]{CK07} there are five such cases:
\setlength{\itemsep}{-1.5pt}
\begin{description}
\setlength{\itemsep}{-1.5pt}
\item $\;$ -- three \emph{simplices} with $N=2$ (a pair antipodal points, the inner product $-1$), $N=3$ (equilateral triangle, the inner product $-\frac12$), and $N=4$ (tetrahedron, the inner product $-\frac13$)
\item $\;$ -- \emph{octahedron} with $N=6$, in other words, three-dimensional cross polytope with the inner products $-1,0$
\item $\;$ -- \emph{icosahedron}, $N=12$, with the inner products $1,\,\pm\frac{1}{\sqrt{5}}$
\end{description}
These are exactly the configurations listed in the theorem which concludes the proof.
\end{proof}

\begin{remarks} \label{remarks}
{\rm (a) The configurations referring to the remaining Platonic solids, cube and dodekahedron with the number of the different inner products $m=3$ and $m=4$, respectively, do not qualify for universality; recall there are neither solutions of the Thomson problem \cite{Twiki}. \\[.2em]
(b) As in Remark~\ref{critcircle} one can check that $\sigma_\mathrm{disc}(H_{\alpha_\mathrm{crit},Y}) \ne \emptyset$ holds for $N=2,3,4,6,12$ and any $Y$ not congruent with the corresponding sharp configuration. \\[.2em]
(c) The optimal configuration listed in the theorem are independent of $\alpha$. This may not be true for other values of $N$ because, in contrast to the Coulomb potential of Thomson's problem, the coupling constant does not enter the inequality as an overall multiplier which could be factorized out. }
\end{remarks}

\section{Back to one dimension: repulsive interactions} \label{s:repulsive}
\setcounter{equation}{0}

While two- and three-dimensional point interactions are always attractive, even if in the latter case it may not be sufficient to produce a discrete spectrum, in one dimension the considerations of Sec.~\ref{s:1Dloop} cover only a part of the problem and one asks what happens if we have $\alpha>0$. The spectrum is again determined by the secular equation which is now $\det \Gamma_{\alpha,Y}(k)=0$ where the resolvent kernel \eqref{1Dgreen} in \eqref{1Dkrein2} is replaced by
\begin{equation} \label{1Dgreen+}
     G_k(x-y) = - \frac{\cos(k(\pi-|x-y|)}{2k \sin \pi k}
\end{equation}
with $k>0$. The spectral threshold $\lambda_1(\alpha,Y)$ is positive and we are again interested in the configuration $\tilde{Y}$ that optimizes it.

In analogy with the modifications of Faber-Krahn inequality for Robin billiards \cite{Ba77, FK15} one might expect that the nature of the optimizer would get switched as $\alpha$ passes through zero, but as we will see it is not the case. What is important, we cannot use the reasoning of the previous sections, for two reasons. One is that that the function \eqref{1Dgreen+} is not convex in general, more exactly it has this property in $(-\pi,\pi)$ only if $k<\frac12$, however, depending on the value of the coupling $\alpha$ the ground state $\lambda_1(\alpha,Y)$ may correspond to any number in $(0,\frac12 N)$. What is even more important, $\lambda_1(\alpha,Y)$ no longer corresponds to the lowest eigenvalue of $\Gamma_{\alpha,Y}(k)$ as the latter has now negative eigenvalues which give rise to higher eigenvalues of $H_{\alpha,Y}$ for $\alpha<0$.

Using a perturbative argument, one can show that the character of the optimizer remains preserved for weak repulsive soupling.

\begin{theorem}\label{thm:1Dweak}
In the same notation as above, we have $\lambda_1(\alpha,Y) \le \lambda_1(\alpha,\tilde{Y})$ for any $N$-point set $Y$ and all $|\alpha|$ small enough, and the inequality is sharp unless $Y$ and $\tilde{Y}$ are congruent.
\end{theorem}
\begin{proof}
We neglect the trivial case of $\alpha=0$ and regard $H_{\alpha,Y}$ as a perturbation of $H_0\equiv H_{0,Y}$ by $V(x)= \alpha \sum_{n=1}^N \delta(x-y_n)$. The spectrum of $H_0$ has been mentioned above, the eigenvalues $m^2$ correspond to the eigenfunctions $\eta_m(x) = \frac{1}{\sqrt{2\pi}} \ee^{imx},\: m\in\Z$. The analytic perturbation theory yields the expansion
\begin{equation} \label{1Dpert}
     \lambda_1(\alpha,Y) = c_1^Y \alpha + c_2^Y \alpha^2 + \OO(\alpha^3)\,,
\end{equation}
where the first-order coefficient $c_1^Y=(\eta_0,V\eta_0) = \frac{N}{2\pi}$ is independent of $Y$ and the second-order one is
\begin{equation} \label{1Dsecond}
     c_2^Y = \sum_{m\ne 0} \frac{|(\eta_m,V\eta_0)|^2}{-m^2} = -\frac{1}{2\pi^2} \sum_{m=1}^\infty \frac{1}{m^2}\,\Big| \sum_{j=1}^N \ee^{imy_j} \Big|^2\,.
\end{equation}
The series obviously converges. Some coefficients may vanish, for instance, if $Y=\tilde{Y}$ and $m=1$ according to \cite[4.4.1.5]{PBM}, however, we will see that some -- in fact many -- are nonzero so the sum is positive.

Our goal is to show that the symmetric configuration $\tilde{Y}$ sharply maximizes the right-hand side of \eqref{1Dsecond}, that is, it minimizes the positive value of the sum appearing there. We rewrite the latter as
$$ 
     \sum_{m=1}^\infty \frac{1}{m^2}\, \sum_{j,j'=1}^N \ee^{im(y_j-y_{j'})} = \sum_{m=1}^\infty \frac{1}{m^2}\, \sum_{j,j'=1}^N \cos(m(y_j-y_{j'}))
$$ 
because the contributions to the imaginary part cancel mutually. We change the order of summation and evaluate the inner series,
$$ 
      \sum_{j,j'=1}^N \sum_{m=1}^\infty \frac{\cos(m(y_j-y_{j'}))}{m^2} = \sum_{j,j'=1}^N \pi^2 B_2\Big(\frac{y_j-y_{j'}}{2\pi}\Big)\,,
$$ 
cf.~\cite[5.4.2.7]{PBM}, where $B_2$ is the Bernoulli polynomial of order two, or explicitly
\begin{align*} \label{}
     \sum_{m=1}^\infty \frac{1}{m^2}\,\Big| \sum_{j=1}^N \ee^{imy_j} \Big|^2 &= \pi^2 \sum_{j,j'=1}^N \Big[ (y_j-y_{j'})^2 - (y_j-y_{j'}) + \frac16 \Big] \\ &= \frac16 \pi^2N^2 + 2\pi^3(N-1) + \pi^2 \sum_{j,j'} (y_j-y_{j'})^2.
\end{align*}
The last term is the only one which depends on the configuration $Y$; using convexity of the quadratic function we can estimate the sum there as follows
\begin{align*} \label{}
     \sum_{j,j'} (y_j-y_{j'})^2 &= \sum_{l=1}^{N-1} \sum_{j=1}^N \big(y_j-y_{j+l\,(\mathrm{mod}\,N)}\big)^2 \ge \sum_{l=1}^{N-1} \Big(\sum_{j=1}^N \big(y_j-y_{j+l\,(\mathrm{mod}\,N)}\big) \Big)^2 \\ &= \frac{1}{N}\,\sum_{l=1}^{N-1} (2\pi l)^2 = \frac23 \pi^2 (N-1)(2N-1) = \sum_{j,j'} (\tilde{y}_j -\tilde{y}_{j'})^2
\end{align*}
with the sharp inequality if $Y$ and $\tilde{Y}$ are not congruent, which concludes the proof.
\end{proof}

Furthermore, a similar result is also valid for strongly repulsive point interactions. In order to demonstrate that, let us introduce the Dirichlet operator $H_{\mathrm{D},Y}$ corresponding to the quadratic form
\begin{equation} \label{1Dirform}
     q_{\mathrm{D},Y}:\; q_{\mathrm{D},Y}(\psi) = \|\psi'\|^2\,, \quad D(q_{\mathrm{D},Y}) = H^1(I\setminus Y)\,,
\end{equation}
in other words, $H_{\alpha,\mathrm{D}}$ acts as the Laplacian with Dirichlet conditions imposed at the points of $Y$. It is straightforward to check that $H_{\alpha,Y} \le H_{\mathrm{D},Y}$, hence, in particular, $\lambda_1(\alpha,Y) \le \lambda_1(\mathrm{D},Y)$ holds by minimax principle. A similar argument using a comparison of quadratic forms \eqref{1Dhamform} for different values of $\alpha$ shows that $\lambda_1(\cdot,Y)$ is an increasing function and we have
\begin{equation} \label{Dirlim}
     \lim_{\alpha\to+\infty} \lambda_1(\alpha,Y) = \lambda_1(\mathrm{D},Y)\,.
\end{equation}

\begin{theorem}\label{thm:1Dstrong}
Suppose that $Y$ and $\tilde{Y}$ are not congruent, then $\lambda_1(\alpha,Y) < \lambda_1(\alpha,\tilde{Y})$ holds for all $\alpha$ large enough.
\end{theorem}
\begin{proof}
The optimization problem for $H_{\alpha,\mathrm{D}}$ is easy to solve because we know the spectrum explicitly, in particular, we have $\lambda_1(\mathrm{D},\tilde{Y}) = \frac14 N^2$ and
\begin{equation} \label{diropt}
 \lambda_1(\mathrm{D},Y) = \bigg( \frac{\pi}{\max\big\{d_{j,j+1\,(\mathrm{mod}\,N)} \big\}} \bigg)^2 < \lambda_1(\mathrm{D},\tilde{Y})
\end{equation}
by assumption. The eigenvalues of $\Gamma_{\alpha,Y}(k)$ entering the secular equation $\det \Gamma_{\alpha,Y}(k)=0$ are real analytic functions of $k$. Combining this fact with the implicit function theorem and \eqref{Dirlim} we infer that
$$ 
      \lambda_1(\alpha,Y) = \lambda_1(\mathrm{D},Y) - c\alpha^{-1} + \OO(\alpha^{-2}) \quad \text{and} \quad \lambda_1(\alpha,\tilde{Y}) = \lambda_1(\mathrm{D},\tilde{Y}) - \tilde{c}\alpha^{-1} + \OO(\alpha^{-2})
$$ 
for some $c,\,\tilde{c}>0$ as $\alpha\to+\infty$; this in combination with \eqref{diropt} yields the sought claim.
\end{proof}

Theorems~\ref{thm:1Dweak} and \ref{thm:1Dstrong} motivate us to the following guess:

\begin{conjecture}\label{1Dconj}
In the described setting, $\lambda_1(\alpha,Y) < \lambda_1(\alpha,\tilde{Y})$ holds for any $\alpha\in(0,\infty)$ unless the $N$-point configurations $Y$ and $\tilde{Y}$ are congruent.
\end{conjecture}

\section{Concluding remarks} \label{s:concl}
\setcounter{equation}{0}

From the physical point of view configurations minimizing energy may be more interesting as they are associated with stability. A maximizer represents an unstable equilibrium -- as the skydiver who allegedly attempted landing at the top of the Gateway Arc learned the hard way. If we deal with attractive interactions, though, no minimizing configurations may even exist; in the present context it is well known that if two point interactions in $\R^\nu,\: \nu=2,3$, approach each other the ground state escapes to negative infinity. At the same time, optimization may be nontrivial even if the interactions are repulsive as the discussion of Sec.~\ref{s:repulsive} shows: evenly spread interactions will \emph{not} minimize the energy if the energy cost of `huddling them together' would be lower.

The above results concerning one-dimensional point interactions can be also regarded as an optimization of the lowest spectral band threshold in periodic systems on the line, because this quantity is associated with the lowest periodic solution, cf.~\cite[Sec.~III.2.3]{AGHH} or \cite{EKW10}.  In a sense this determines a maximizer for point interactions distributed on the line with fixed density, but this claim cannot taken literally, rather in the way one treats usually the thermodynamic limit, taking a finite system with periodic boundary conditions and letting its side to go to infinity.

The hardest question coming from this discussion concerns without any doubt the sphere optimization. One might conjecture that for large $N$ and strong point interactions the optimal patter would be approximately hexagonal, however, it is not likely to be the case in the weak coupling, that is, for large positive $\alpha$. Generally speaking, any result going beyond the universally optimal configurations considered here should be considered a success.

\section*{Acknowledgments}

The author is obliged to Sylwia Kondej for a useful discussion, to the referees for careful reading of the manuscript, and to Institut Mitag-Leffler, where a part of the work was done, for the hospitality. The research was supported by the Czech Science Foundation (GA\v{C}R) within the project 17-01706S and by the EU project CZ.02.1.01/0.0/0.0/16\textunderscore 019/0000778.


\section*{References}

\begin{thebibliography}{10}
\bibitem[AGHH]{AGHH}
S.~Albeverio, F.~Gesztesy, R.~H{\o}egh-Krohn, H.~Holden: \emph{Solvable Models in Quantum Mechanics}, 2nd edition,  AMS Chelsea Publishing, Providence, R.I., 2005.
\bibitem[AFK17]{AFK17}
P. Antunes, P. Freitas, D. Krej\v{c}i\v{r}\'{\i}k: Bounds and extremal domains for Robin eigenvalues with negative boundary parameter, \emph{Adv. Calc. Var.} \textbf{10} (2017), 357--380.
\bibitem[BB09]{BB09}
E.~Bannai, E.~Bannai: A survey on spherical designs and algebraic combinatorics on spheres, \emph{Eur. J. Combin.} \textbf{30} (2009), 1392--1425.
\bibitem[Ba77]{Ba77}
M.~Bareket: On an isoperimetric inequality for the first eigenvalue of a boundary value problem, \emph{SIAM J. Math. Anal.} \textbf{8} (1977), 280--287.
\bibitem[CK07]{CK07}
H.~Cohen, A.~Kumar: Universally optimal distribution of point on the sphere, \emph{J. Amer. Math. Soc.} \textbf{20} (2007), 99--148.
\bibitem[Ex06]{Ex06}
P. Exner: Necklaces with interacting beads: isoperimetric problems, in \emph{Proceedings of the ``International Conference on Differential Equations and Mathematical Physics'' (Birmingham 2006)}, AMS ``Contemporary Math" Series, vol.~412, Providence, R.I., 2006; pp.~141--149.
\bibitem[EHL06]{EHL06}
P.~Exner, E.M.~Harrell, M.~Loss: Inequalities for means of chords, with application to isoperimetric problems, \emph{Lett. Math. Phys.} \textbf{75} (2006), 225--233; addendum \textbf{77} (2006), 219.
\bibitem[EK19]{EK19}
P.~Exner, S.~Kondej: An optimization problem for leaky star graphs of codimension two, \emph{in preparation}
\bibitem[EKW10]{EKW10}
P.~Exner, P.~Kuchment, B.~Winn: On the location of spectral edges in $\mathbb{Z}$-periodic media, \emph{J. Phys. A: Math. Theor.} \textbf{43} (2010), 474022
\bibitem[EL17]{EL17}
P.~Exner, V.~Lotoreichik: A spectral isoperimetric inequality for cones, \emph{Lett. Math. Phys.} \textbf{107} (2017), 717--732.
\bibitem[EL18]{EL18}
P.~Exner, V.~Lotoreichik: Optimization of the lowest eigenvalue for leaky star graphs, in Proceedings of the conference
``Mathematical Results in Quantum Physics'' (QMath13, Atlanta 2016; F.~Bonetto, D.~Borthwick, E.~Harrell, M.~Loss, eds.), \emph{Contemporary Math.}, vol.~717, AMS, Providence, R.I., 2018; pp.~187--196.
\bibitem[Fa23]{Fa23}
G. Faber: Beweis, dass unter allen homogenen Membranen von gleicher Fl\"ache und gleicher Spannung die kreisf\"ormige den tiefsten Grundton gibt, \emph{Sitzungsber. Bayer. Akad. Wiss. München, Math.-Phys. Kl.} (1923), 169--172.
\bibitem[FK15]{FK15}
P.~Freitas, D. Krej\v{c}i\v{r}\'{\i}k: The first Robin eigenvalue with negative boundary parameter, \emph{Adv. Math.} \textbf{280} (2015), 322--339.
\bibitem[Kr25]{Kr25}
E. Krahn: \"Uber eine von Rayleigh formulierte Minimaleigenschaft des Kreises, \emph{Math. Ann.} \textbf{94} (1925), 97--100.
\bibitem[Kr53]{Kr53}
M.G.~Krein: On the trace formula in perturbation theory, \emph{Mat. Sb.} \textbf{53} (1953), 597--626.
\bibitem[KL18]{KL18}
D. Krej\v{c}i\v{r}\'{\i}k, V.~Lotoreichik: Optimisation of the lowest Robin eigenvalue in the exterior of a compact set, \emph{J. Convex Anal.} \textbf{25} (2018), 319--337.
\bibitem[KL19]{KL19}
D. Krej\v{c}i\v{r}\'{\i}k, V.~Lotoreichik: Optimisation of the lowest Robin eigenvalue in the exterior of a compact set, II: non-convex domains and higher dimensions, \emph{Potential. Anal.}, to appear; \texttt{arXiv:1707.02269 [math.SP]}
\bibitem[PBM]{PBM}
A.P.~Prudnikov, Yu.A.~Brychov, O.I.~Marichev: \emph{Integrals and Series}, Nauka, Moscow 1981.
\bibitem[Sch13]{Sch13}
R.E.~Schwartz: The five-electron case of Thomson's problem, \emph{Experim. Math.} \textbf{22} (2013), 157--186.
\bibitem[Sm98]{Sm98}
S.~Smale: Mathematical Problems for the Next Century, \emph{Math. Intelligencer} \textbf{2} (1998), 7--15.
\bibitem[Th04]{Th04}
J.J.~Thomson: On the structure of the atom: an investigation of the stability and periods of oscillation of a number of corpuscles arranged at equal intervals around the circumference of a circle; with application of the results to the theory of atomic structure, \emph{Phil. Mag.} \textbf{7} (1904), 237--265.
\bibitem[Twiki]{Twiki}
\emph{https://en.wikipedia.org/wiki/Thomson$\underline{\phantom{n}}$problem}


\end{thebibliography}
\end{document}